\documentclass[10pt,onecolumn]{IEEEtran}

\usepackage{graphics} 
\usepackage{epsfig} 

\usepackage{times,verbatim} 
\usepackage{amsmath, amssymb}
\usepackage{amsthm}
\usepackage{thmtools}
\usepackage{graphicx}
\usepackage{subfigure}
\usepackage{epstopdf}
\usepackage{algorithm}
\usepackage{algorithmic}
\usepackage{enumerate}
\usepackage{color}

\newtheorem{theorem}{Theorem}[section]
\newtheorem{lemma}[theorem]{Lemma}

\newtheorem{problem}[theorem]{Problem}
\newtheorem{proposition}[theorem]{Proposition}

\newtheorem{definition}[theorem]{Definition}

\numberwithin{equation}{section}

\newcommand{\R}{{\mathbb{R}}}

\newcommand{\B}{{\mathbb B}}

\newcommand{\N}{{\mathbb{N}}}

\newcommand{\norm}[1]{\left\Vert#1\right\Vert_2}
\renewcommand{\matrix}[1]{\begin{bmatrix}#1\end{bmatrix}}

\newcommand{\supp}{\textrm{supp}}

\usepackage{tikz}
\usetikzlibrary{shapes,shadows,calc,snakes}
\usetikzlibrary{decorations.pathmorphing,patterns}
\usetikzlibrary{decorations}
\usetikzlibrary{decorations.mypathmorphing}
\usepackage{pgfplots}

\IEEEoverridecommandlockouts

\title{\LARGE \bf A Satisfiability Modulo Theory Approach to Secure State Reconstruction  in Differentially Flat Systems Under Sensor Attacks}

\author{Yasser Shoukry,  Pierluigi Nuzzo, Nicola Bezzo, \\
Alberto L. Sangiovanni-Vincentelli, Sanjit A. Seshia, and Paulo Tabuada
\thanks{Y. Shoukry and P. Tabuada are with the Electrical
Engineering Department, U.C. Los Angeles, {\tt\footnotesize \{yshoukry, tabuada\}@ucla.edu}}%
\thanks{P. Nuzzo, A. L. Sangiovanni-Vincentelli, and S. A. Seshia are with the Department of Electrical Engineering and Computer Sciences, U.C. Berkeley, {\tt\footnotesize \{nuzzo,alberto,sseshia\}@eecs.berkeley.edu}}
\thanks{N. Bezzo is with the PRECISE Center, University of Pennsylvania, {\tt\footnotesize nicbezzo@seas.upenn.edu}}%
\thanks{This work was partially sponsored by the NSF award 1136174, by DARPA under agreement number FA8750-12-2-0247, by TerraSwarm, one of six centers of STARnet, a Semiconductor Research Corporation program sponsored by MARCO and DARPA, and by the NSF project ExCAPE: Expeditions in Computer Augmented Program Engineering (award 1138996). The U.S. Government
is authorized to reproduce and distribute reprints for Governmental purposes
notwithstanding any copyright notation thereon. The views and conclusions
contained herein are those of the authors and should not be interpreted
as necessarily representing the official policies or endorsements, either
expressed or implied, of NSF, DARPA or the U.S. Government.
}
}

\begin{document}

\maketitle

\begin{abstract}
We address the problem of estimating the state of a differentially flat system from measurements that may be corrupted by an adversarial attack. In cyber-physical systems, malicious attacks can directly compromise the system's sensors or manipulate the communication between sensors and controllers. 
We consider attacks that only corrupt a subset of sensor measurements. We show that the possibility of reconstructing the state under such attacks is characterized by a suitable generalization of the notion of s-sparse observability, previously introduced by some of the authors in the linear case. We also extend our previous work on the use of Satisfiability Modulo Theory solvers to estimate the state under sensor attacks to the context of differentially flat systems. The effectiveness of our approach is illustrated on the problem of controlling a quadrotor under sensor attacks.
%
%
%
\end{abstract}

\section{Introduction}

The increasing dependence on sensors and cyber components (e.g. digital processors and networks) to monitor and control a broad range of today's critical infrastructures is at the outset of unprecedented vulnerabilities and malicious attacks. A striking example of such attacks is the Stuxnet virus targeting SCADA systems \cite{stuxnet}. In this attack, sensor measurements are replaced by previously recorded data, which, once they are fed to the controller, can lead to catastrophic situations. Other examples include the injection of false data in ``smart'' systems \cite{LiuPowerAttack}, and the non-invasive sensor spoofing attacks in the automotive domain \cite{YasserABS}.

To secure these cyber-physical systems, 
a possible strategy is to exploit 
an accurate mathematical model of the dynamics of the physical system under control, and analyze any discrepancy between the actual sensor measurements and the ones predicted by the model, to decide about the existence of an adversarial attack~\cite{Bullo_CSM,Hamza_TAC}. Once the malicious sensors, if any, are detected and identified, it is then possible to estimate the actual system state by using the data collected from the attack-free sensors. In the following, we refer to this approach as \emph{secure state reconstruction}.

The problem of state reconstruction in the presence of disturbances, in its general form, has attracted considerable attention from the control community over the years. Previous work addresses the problem in terms of robust filter (estimator) design against outliers~\cite{BoydKF, Mitter, Georgios_TSP2}. However, the lack of a priori knowledge about the attack signals tends to limit the applicability of robust estimation techniques  to security problems. In secure state reconstruction, no assumptions are usually made about the attacks, e.g.,~in terms of their stochastic properties, time evolution, or energy bounds.  
%

A game theoretic formulation for the secure state reconstruction problem has been proposed in the literature, when the physical system is scalar~\cite{Sinopoli_TAC2}. An alternative reconstruction technique, still in the context of a scalar system equipped with one sensor, has also been derived based on the analysis of the performance degradation of a Kalman filter when the sensor is under attack~\cite{Bai_Gupta}. Finally, the general case of a multidimensional system equipped with multiple sensors has been tackled~\cite{Bullo_TAC, Hamza_TAC, YasserETPGarXiv, YasserCDC_ET, Pajic_ICCPS, Yasser_SMT_ACC, Yasser_SMT, Joao_ACC, Kalle_TAC} when the attackers are restricted to corrupt an unknown subset of the system sensors. However, all of the above contributions focus on problems for which the underlying dynamics can be described by a linear system. 

Unlike previous work, we focus in this paper on the problem of
secure state reconstruction for a class of nonlinear systems. Specifically, we consider physical systems whose dynamics can be described by a differentially flat system~\cite{Fliess95flatnessand}. Differentially flat systems represent an important class of nonlinear systems, in that they encompass a wide range of mechanical systems, including several examples of ground and aerial vehicles.

While differentially flat systems can be converted into linear systems using dynamic feedback linearization and a change of coordinates, this technique would, however, require the knowledge of the system state. Since this is clearly our ultimate goal, it is not possible to directly apply the results from linear secure state reconstruction to differentially flat systems. We follow instead a different approach, by extending the notion of $s$-sparse observability~\cite{YasserETPGarXiv, YasserCDC_ET} from linear systems to 
nonlinear systems. Similarly to linear systems, we show that $s$-sparse observability provides a necessary and sufficient condition for the full reconstruction of the state regardless of the attack. Resting on this concept, we can then build on our previous work on sound and complete secure state reconstruction for linear systems~\cite{Yasser_SMT_ACC,Yasser_SMT}, to develop an algorithm that can efficiently identify the corrupted sensors by leveraging Satisfiability Modulo Theory (SMT) solving~\cite{smtbook} to tackle the combinatorial aspects of the problem. We illustrate our algorithm on the problem of controlling and stabilizing a quadrotor, while some of its sensors are under attack.

The rest of this paper is organized as follows. Section~\ref{sec:problem} formulates the secure state reconstruction problem. In Section~\ref{sec:obs}, we generalize the notion of $s$-sparse observability to nonlinear systems and then show that it is the necessary and sufficient condition to reconstruct the state in spite of the attack. Section~\ref{sec:theory} presents the generalization of our previous SMT-based attack detection and reconstruction algorithm to differentially flat systems. In Section~\ref{sec:results}, we discuss its application to the quadrotor case study. We finally draw some conclusions in Section~\ref{sec:conclusion}. 

%
\section{The Secure State Reconstruction Problem}
\label{sec:problem}

\subsection{Notation}
The symbols $\N, \R$ and $\B$ denote the sets of natural, real, and Boolean numbers, respectively. 
If $S$ is a set, we denote by $|S|$ its cardinality. The support of a vector $x\in \R^n$, 
denoted by $\supp(x)$, is the set of indices of the non-zero components of $x$. Similarly, the complement of the support of a vector $x$ is denoted by $\overline{\supp(x)} = \{1,\hdots,n\}\setminus\supp(x)$. We call a vector $x \in \R^n$ $s$-sparse, if  $x$ has $s$ nonzero
elements, i.e., if  $\vert \supp(x)\vert=s$. 

Let $f: \R^n \rightarrow \R^m$ be a function given by $f(x)=(f_1(x),\hdots,f_m(x))$, where  $f_i : \R^n \rightarrow \R$ is the $i${th} component of $f$. Then, for the set $\Gamma \subseteq \{1, \hdots, m\}$, we denote by
$f_{\Gamma}$ the vector function obtained from $f$ by removing all the components except those indexed by $\Gamma$. 
Similarly, $f_{\overline{\Gamma}}$ is obtained from $f$ by removing the components indexed by $\Gamma$. 
Finally, we use the notation $\nabla_x f$ to denote the Jacobian matrix of $f$ evaluated at $x$.



\subsection{Dynamics and Attack Model}

We consider a system of the form:
\begin{align}
\Sigma_a \quad
\begin{cases}
	x^{(t+1)} &= f\left(x^{(t)},  u^{(t)}\right), \\
	y^{(t)} &= h\left(x^{(t)} \right) + a^{(t)}
\end{cases}
\label{eq:system_attack}
\end{align} 
where $x^{(t)} \in \mathcal{X} \subseteq \R^n$ is the system state, $u^{(t)} \in \mathcal{U} \subseteq \R^m$ is the system input, and $y^{(t)} \in \R^{p}$ is the observed output, all at time  $t \in \N$. The map
$f: \mathcal{X} \times \mathcal{U}  \rightarrow \mathcal{X}$ 
represents the system dynamics.  We will use the notation $f_u(x) = f(x,u)$ in the remainder of this paper. We also use the notation $f_{u_k u_{k-1} \ldots  u_{1}}(x)$ to denote the $k$-fold composition of $f$, i.e., 
$$ f_{u_k u_{k-1} \ldots  u_{1}}(x)= f\left(f\left(f\left(f\left(u^{(1)}, x\right), \ldots\right), u^{(k-1)}\right), u^{(k)}\right).$$

An attacker corrupts the sensor measurements $y$ by either spoofing the sensor outputs, or by manipulating the data transmitted from the sensors to the controller. Independently of how the attack is implemented, its effect can be described by the $s$-sparse vector $a^{(t)} \in \R^p$.  If sensor $i \in \{1, \hdots ,p\}$ is attacked then
the $i${th} component of $a^{(t)}$ is non-zero; otherwise the $i${th}
sensor is not attacked. Hence, $s$ describes the number of attacked sensors. 
We make no assumptions on the vector $a^{(t)}$ other than being $s$-sparse. 
In particular, we do not assume bounds, statistical properties, or restrictions on 
the time evolution of the elements in $a^{(t)}$. While the value of $s$ is not
known, we assume the knowledge of an upper bound $\overline{s}$ on it.

\subsection{Problem Formulation}
\label{sec:pform}

Solving the secure state reconstruction problem implies estimating the state $x$ from a set of measurements collected over a window of length $\tau\in \N$. Hence,  we start by grouping the measurements from the $i$th  sensor as:
\begin{align}
Y_i^{(t)} = H_{u,{i}}\left(x^{(t-\tau+1)}\right) + E_i^{(t)} 
\label{eq:output_eqn}
\end{align}
where:
\begin{align*}
	Y_i^{(t)} &= \matrix{
		y_i^{(t-\tau+1)} \\ y_i^{(t-\tau)} \\ \vdots \\ y_i^{(t)}
	}, 
	E_i^{(t)}  =\matrix{ 
		a_i^{(t-\tau+1)} \\ a_i^{(t-\tau)} \\ \vdots \\ a_i^{(t)}
	},
	H_{u,{i}} \left(x^{(t-\tau+1)}\right) = \matrix{
		h_i\left(x^{(t-\tau+1)}\right)  \\ h_i \left(f_{u_{t-\tau+1}} \left(x^{(t-\tau+1)}\right)\right) \\ \vdots \\  h_i \left(f_{u_{t}  \ldots u_{t-\tau+1}} \left(x^{(t-\tau+1)}\right)\right)
	}.
\end{align*}
We then define:
$$ Y^{(t)} = \matrix{Y_1^{(t)} \\ \vdots \\ Y_p^{(t)}}, \quad E^{(t)} = \matrix{E_1^{(t)} \\ \vdots \\ E_p^{(t)}}, \quad H_{u} = \matrix{H_{u,{1}} \\ \vdots \\ {H}_{u,{p}}}$$
where, with some abuse of notation, $Y_i, E_i$ and ${H}_{u,{i}}$
are used to denote the $i$th block of $Y^{(t)}, E^{(t)}$ and ${H}_{u}$, respectively. Using the same notation, we denote by $Y_{\Gamma}, E_{\Gamma}$, and ${H}_{u,{\Gamma}}$ the blocks indexed by the elements in the set $\Gamma$. 
%
Moreover, for simplicity, we drop the time $t$ argument in the following, since we assume that the secure state reconstruction problem is to be solved at every time instance.

Let $(x^*, E^*)$ denote the actual state of the system and the actual attack vector. Let also $b^* \in \B^{p}$ be a vector of binary indicator variables such that $b^*_i = 0$ when the $i$th sensor is attack-free and $b_i = 1$ otherwise. It follows from~\eqref{eq:output_eqn} that:
\begin{align*}
Y_i = 
\begin{cases}
H_{u,i}(x^*) & \text{if $b^*_i = 0$}\\
H_{u,i}(x^*) + E^*_i & \text{if $b^*_i = 1$}.
\end{cases}
\end{align*}
Therefore, we are interested in a state estimate $x$ and a vector of binary indicator variables $b = (b_1, \ldots, b_p)$ such that the discrepancy between the collected measurements $Y_i$ and the expected outputs $H_{u,i}(x)$ is zero for all the sensors that are labeled as attack-free sensors ($b_i^* = 0$). Furthermore, the estimated state $x$ should be equal to $x^*$. These requests can be formalized as follows.
\begin{problem}{\textbf{(Secure State Reconstruction)}}
For the control system under attack $\Sigma_a$ (defined in~\eqref{eq:system_attack}), construct the estimate \mbox{$\eta = (x, b ) \in \R^n \times \B^p$} such that $\eta \models \phi$ (i.e., $\eta$ satisfies the formula $\phi$),  where:
\begin{align}
\phi ::= & \bigwedge_{i = 1}^p \Bigg( \lnot b_i  \Rightarrow \norm{Y_i - {H}_{u,{i}}( x )}^2 = 0 \Bigg)
 \bigwedge \left( \sum_{i}^p b_i \le \overline{s} \right), 
 \label{eq:phi}
\end{align}
\label{prob:sse}
\end{problem}
subject to $(x^* = x)  \land (\supp(b^*) \subseteq \supp(b)).$

The second clause in the formula $\phi$ rules out the trivial solution in which all sensors are labelled as attacked, by enforcing a cardinality constraint on the number of attacked sensors, which is required to be bounded by $\overline{s}$. 

As in the case of linear systems~\cite{Yasser_SMT}, the secure state reconstruction problem formulation in Problem~\ref{prob:sse} does not ask for a solution with the minimal number of attacked sensors. However, as shown in \cite{Yasser_SMT}, it is possible to obtain the minimal set of sensors under attack by invoking a solver for Problem \ref{prob:sse} multiple times.
In the next section, we characterize the conditions that guarantee the existence of a solution for this problem. 

\section{$s$-Sparse Observability}
\label{sec:obs}

For linear systems, the notion of $s$-sparse observability plays a key role in determining the existence of a solution for Problem~\ref{prob:sse}~\cite{YasserETPGarXiv}. In this section, we generalize this notion to the case  of nonlinear systems. To do so, we consider an attack-free discrete-time nonlinear system of the form:
\begin{align}
\Sigma  \quad
\begin{cases}
	x^{(t+1)} &= f\left(x^{(t)}, u^{(t)}\right),  t \in \N\\
	y^{(t)} &= h\left(x^{(t)}\right)
\end{cases}
\label{eq:system}
\end{align} 
%
and recall the general definitions of indistinguishability and observability.

\begin{definition}[Indistinguishability]
We say that two states $x,x' \in \mathcal{X}$ of system $\Sigma$ are \emph{indistinguishable} from measurements collected from the set of sensors indexed by $\Gamma$ over a window of length $\tau$, and we write $x \ I_{\Gamma}^{\tau} \ x'$ if, for every sequence of controls $u^{(1)}, \hdots, u^{(\tau)} \in
\R^m$ we have $ H_{u,\Gamma} (x) = H_{u,\Gamma}  (x') $.
\end{definition}
\begin{definition}[Observability]
	A state $x \in \mathcal{X}$ of system $\Sigma$ is said to be \emph{observable} using measurements collected from the set of sensors indexed by $\Gamma$ over a window of length $\tau$ \big($(\tau,\Gamma)$-observable for short\big), if, for each $x' \in \mathcal{X}$,  $x \ I_{\Gamma}^{\tau} \ x'$ implies $x = x'$.
\end{definition}
\begin{definition}
	A system $\Sigma$ is said to be \emph{$(\tau,\Gamma)$-observable} if all the states $x \in \mathcal{X}$ are $(\tau,\Gamma)$-observable.
\end{definition}

We can now define the notion of $s$-sparse observability as follows.
\begin{definition}[\textbf{s-Sparse $\mathbf{\tau}$-Observable System}]
The nonlinear control system $\Sigma$, defined in~\eqref{eq:system},
is said to be $s$-sparse $\tau$-observable
if for every set $\Gamma\subseteq\{1,\hdots,p\}$ with $|\Gamma| = s$, the system $\Sigma$
 is $(\tau,\overline{\Gamma})$-observable, where $\overline{\Gamma} = \{1,\ldots,p\}\setminus \Gamma$. 
\end{definition}
In other words, a system is $s$-sparse $\tau$-observable if it is $\tau$-observable from any choice of $(p  - s)$ sensors.

\subsection{$s$-Sparse Observability and Secure State Reconstruction}

Similarly to linear systems, we show that $2\overline{s}$-sparse observability is a necessary and sufficient condition for the existence of a solution of the secure state reconstruction problem. To do so, we start by showing an intermediate result, stating that the correct reconstruction of the system state implies correct estimation of the attack support. To simplify the notation, we also drop the subscript $u$ in all of the proofs.

\begin{proposition} 
For any pair $\eta = (x,b)$ that satisfies $\phi$ in~\eqref{eq:phi} (i.e., such that $\eta \models \phi$), the following holds:
$$ x^* = x \Rightarrow \supp(b^*) \subseteq \supp(b).$$
\label{prop:x_b}
\end{proposition}

\begin{proof}
The result follows from the fact that the first set of clauses in $\phi$ enforces:
\begin{align*}
 \bigwedge_{i \in \overline{\supp(b)}} \norm{Y_i - {H}_{i}( x )}^2 = 0 
  &\Rightarrow \norm{Y_{\overline{\supp(b)}} - {H}_{\overline{\supp(b)}}(x)}^2 = 0 \\
  &\Rightarrow \norm{{H}_{\overline{\supp(b)}}(x^*) + E^*_{\overline{\supp(b)}} - {H}_{\overline{\supp(b)}}(x)}^2 = 0 \\
  &\Rightarrow \norm{E^*_{\overline{\supp(b)}}}^2 = 0
 \end{align*}
 where the last implication follows from $x$ and $x^*$ being equal. The last equality implies that the attack vector $E^*$, after removing the sensors indexed by the estimate set $\supp(b)$, is equal to zero; hence, we conclude that \mbox{$\supp(b^*) \subseteq \supp(b)$}.
%
\end{proof}

Now we can state the main result of this section as follows.

\begin{theorem} 
For any pair $\eta = (x,b)$ that satisfy $\phi$ in~\eqref{eq:phi} the following holds:
$$ x^* = x \qquad \land \qquad \supp(b^*) \subseteq \supp(b)$$
if and only
if the dynamical system $\Sigma_a$ defined by \eqref{eq:system_attack} is $2\overline{s}$-sparse $\tau$-observable.
\label{th:sse}	
\end{theorem}

\begin{proof}
We need to show that $2\overline{s}$-sparse $\tau$-observability is a necessary and sufficient condition for $x = x^*$ to hold. Once this is established, $\supp(b^*) \subseteq \supp(b)$ follows from Proposition \ref{prop:x_b}.

We assume for the sake of contradiction that there exists a pair $\eta = (x,b)$ that satisfies $\phi$, and such that $x \ne x^*$, while the system $\Sigma_a$ is $2\overline{s}$-sparse $\tau$-observable. Then, 
the first set of clauses in $\phi$ implies that:
\begin{align}
 \bigwedge_{i \in \overline{\supp(b)}} \norm{Y_i - {H}_{i}( x )}^2 = 0 
 & \Rightarrow \norm{Y_{\overline{\supp(b)}} - {H}_{\overline{\supp(b)}}(x)}^2 = 0  \nonumber \\
 & \Rightarrow  \norm{Y_{\overline{\supp(b) \cup \supp(b^*)}} - {H}_{\overline{\supp(b) \cup \supp(b^*)}}(x)}^2 = 0 \nonumber \\
 & \stackrel{(a)}{\Rightarrow}  \norm{{H}_{\overline{\supp(b) \cup \supp(b^*)}}(x^*) - {H}_{\overline{\supp(b) \cup \supp(b^*)}}(x)}^2 = 0 \nonumber  \\
& \Rightarrow  {H}_{\overline{\supp(b) \cup \supp(b^*)}}(x^*) = {H}_{\overline{\supp(b) \cup \supp(b^*)}}(x) 
\label{eq:proof_equality}
\end{align}
where the implication $(a)$ follows by using the equality~\eqref{eq:output_eqn}, along with the fact that $E^*_{\overline{\supp(b^*)}}$ being zero implies that $E^*_{\overline{\supp(b) \cup \supp(b^*)}}$ is also equal to zero.

However, by the second clause in $\phi$ we know that the cardinality of ${\supp(b)}$ is at most $\overline{s}$, hence the cardinality of $\vert {\supp(b) \cup \supp(b^*)} \vert$ is at most $ 2\overline{s}$. Equality~\eqref{eq:proof_equality} would then imply that $x$ and $x^*$ are indistinguishable using $p - 2 \overline{s}$ sensors, which in turn implies that the system $\Sigma_a$ is not $2 \overline{s}$-sparse $\tau$-observable, a contradiction with respect to our original assumption.

Conversely, if the system $\Sigma_a$ is not $2s$-sparse $\tau$-observable then there exists a vector $x$  with $x \ne x^*$ which is indistinguishable from  $x^*$, when using $p - 2\overline{s}$ sensors. This in turn implies that 
${H}_{\Gamma}(x^*)  ={H}_{\Gamma}(x)$ for some $\Gamma$ that have cardinality $p - 2\overline{s}$. Hence, we can define two indicator variables $b$ and $b^*$ such that $\Gamma=\overline{\supp(b)\cup\supp(b^*)}$, $\vert\supp(b)\vert = \overline{s}$ and $\vert \supp(b^*)\vert = \overline{s}$. We can also define the attack vector $E^*$ as :
\begin{align*} 
E^*_{\supp(b^*)} &= {H}_{\supp(b^*)}(x)  - {H}_{\supp(b^*)}(x^*). 
\end{align*}
Then, using the above definitions, together with the assumption ${H}_{\Gamma}(x^*)  ={H}_{\Gamma}(x)$, we obtain:
\begin{align*}
\norm{Y_{\overline{\supp(b)}} -  {H}_{\overline{\supp(b)}}(x)  }^2 
&\quad = \norm{{H}_{\overline{\supp(b)}}(x^*) +  E^*_{\overline{\supp(b)}} -  {H}_{\overline{\supp(b)}}(x)  }^2 \\
&\quad = \norm{\matrix{{H}_{\Gamma}(x^*) \\ {H}_{\supp(b^*)}(x^*)} + \matrix{0 \\ E^*_{\supp(b^*)}} - \matrix{{H}_{\Gamma}(x) \\ {H}_{\supp(b^*)}(x)}}^2 \\
&\quad = 0
\end{align*}
which implies that there exists an estimate $\eta=(x,b)$, with $x \ne x^*$ that also satisfies $\phi$.
\end{proof}

\subsection{Differential Flatness and $s$-Sparse Flatness}
\label{sec:flatness}
For the rest of this paper, we consider a special class of nonlinear systems known as \emph{differentially flat}\footnote{Although the term  \emph{difference flatness} is sometimes used in the literature for systems governed by difference equations, we choose to employ the widely accepted term \emph{differential flatness}.} systems. A system is differentially flat if the state and the input can be reconstructed from current and previous outputs. More formally,

\begin{definition}[Differential Flatness]
System is differentially flat if
there exist 
an integer $k \in \N$,  and functions $\alpha$ and $\beta$, such that the state and the input can be reconstructed from the outputs $y^{(t)}$ as follows:
\begin{align}
x^{(t)} &= \alpha \left(y^{(t)}, y^{(t-1)}, \ldots, y^{(t-k+1)} \right) \label{eq:flatObserver}\\
u^{(t)} &= \beta \left(y^{(t)}, y^{(t-1)}, \ldots, y^{(t-k+1)} \right) .
\end{align}
In such case, the output $y^{(t)}$ is called a flat output.
\end{definition}
In the remainder of this paper, we assume that the window length $\tau$ in~\eqref{eq:output_eqn} is chosen such that $\tau = k$.

\begin{definition}[\textbf{s-Sparse Flat System}]
The nonlinear control system $\Sigma_a$, defined by~\eqref{eq:system_attack},
is said to be $s$-sparse flat
if for every set $\Gamma\subseteq\{1,\hdots,p\}$ with $|\Gamma| = s$, the system $\Sigma_{\overline{\Gamma}} $:
\begin{align}
\Sigma_{\overline{\Gamma}}   \quad
\begin{cases}
	x^{(t+1)} &= f\left(x^{(t)}, u^{(t)}\right), \\
	y^{(t)} &= h_{\overline{\Gamma}} \left(x^{(t)}\right)
\end{cases}
\end{align} 
is differentially flat.
\end{definition}
In other words, the system is $s$-sparse flat if any choice of $p  - s$ sensors is  a flat output. It is then straightforward to show that $s$-sparse flatness implies  $s$-sparse $\tau$-observability.

\section{Secure State Reconstruction Using SMT Solving}
\label{sec:theory}
The secure state reconstruction problem is combinatorial, since a direct solution would require constructing the state from all different combinations of $p - \overline{s}$ sensors to determine which sensors are under attack. In this section, we show how using SMT solving can dramatically reduce the complexity of the reconstruction algorithm.

To decide whether the combination of Boolean and nonlinear constraints in~\eqref{eq:phi} is satisfiable, we develop the detection algorithm \textsc{Imhotep-SMT} using the \emph{lazy} SMT paradigm~\cite{smtbook}. By building upon the \textsc{Imhotep-SMT} solver~\cite{Yasser_SMT,Imhotep_SMT_CAV}, our 
decision procedure combines a SAT solver (\textsc{SAT-Solve}) and a theory solver ($\mathcal{T}$-\textsc{Solve}). 
However, differently than~\cite{Yasser_SMT, Imhotep_SMT_CAV}, the theory solver in this paper can also reason about the nonlinear constraints in~\eqref{eq:phi}, as generated from a differentially flat system.
The SAT solver efficiently reasons about combinations of Boolean and pseudo-Boolean constraints\footnote{A pseudo-Boolean constraint is a linear constraint over Boolean variables with integer coefficients.}, using the 
David-Putnam-Logemann-Loveland (DPLL) algorithm~\cite{lazySMT} to suggest possible assignments for the nonlinear constraints. The theory solver checks the consistency of the given assignments, and provides the reason for the conflict, a \emph{certificate}, or a counterexample, whenever inconsistencies are found. Each certificate results in learning new constraints which will be used by the SAT solver to prune the search space. The complex decision task is thus broken into two simpler tasks, respectively, over the Boolean and nonlinear domains. 


\subsection{Overall Architecture}
As illustrated in Algorithm~\ref{alg:smt},  we start by mapping each nonlinear constraint to an auxiliary Boolean variable $c_i$ to obtain the following (pseudo-)Boolean satisfiability problem:
%
\begin{align*}
\phi_B :=  \bigwedge_{i = 1}^p  \Bigg(\lnot b_i  \Rightarrow c_i \Bigg) \bigwedge \left( \sum_{i = 1}^p b_i \le \overline{s} \right)
\end{align*}
where $c_i = 1$ if $\norm{Y_i - {H}_{u,{i}} (x) } = 0$ is satisfied, and zero otherwise. By only relying on the Boolean structure of  the problem, \textsc{SAT-Solve} returns an assignment for the variables $b_i$ and $c_i$ (for $i = 1, \hdots, p$), thus hypothesizing which sensors are attack-free, hence which nonlinear constraints should be jointly satisfied.  
This Boolean assignment is then used by $\mathcal{T}$-\textsc{Solve} to determine whether there exists a state $x \in \R^n$ which satisfies all the nonlinear constraints related to the unattacked sensors, i.e.~$\{ \norm{Y_i - {H}_{u,{i}} ({x}) } = 0 | i \in \overline{\supp(b)} \}$ is the set of constraints sent to $\mathcal{T}$-\textsc{Solve}. If $x$ is found, \textsc{Imhotep-SMT} terminates with $\verb+SAT+$ and provides the solution $(x,b)$. Otherwise, the $\verb+UNSAT+$ certificate $\phi_{\text{cert}}$ is generated in terms of new Boolean constraints, explaining which sensor measurements are conflicting and may be under attack. A na\"{i}ve certificate can always be generated in the form of:
\begin{align} \label{eq:trivcert} 
 \phi_{\text{triv-cert}} = \sum_{i \in \overline{\supp(b)}} b_i \ge 1,
\end{align}
which encodes the fact that at least one of the sensors in the set $\overline{\supp(b)}$ (i.e. for which $b_i=0$ in the current iteration) is actually under attack, and must be set to one in the next assignment of the SAT solver. The augmented Boolean problem is then fed back to \textsc{SAT-Solve} to produce a new assignment, and the sequence of new SAT queries repeats until $\mathcal{T}$-\textsc{Solve} terminates with $\verb+SAT+$.

By assuming that the system is $2\overline{s}$-sparse flat, it follows from Theorem \ref{th:sse} that there always exists a solution to Problem~\ref{prob:sse}, hence Algorithm~\ref{alg:smt} will always terminate. However, to help the SAT solver quickly converge towards the correct assignment, a central problem in lazy SMT solving is to generate succinct explanations whenever conjunctions of nonlinear constraints are unfeasible. 

The rest of this section will then focus on the implementation of the two main tasks of $\mathcal{T}$-\textsc{Solve}, namely, (i) checking the satisfiability of a given assignment ($\mathcal{T}$-\textsc{Solve.Check}), and (ii) generating succinct UNSAT certificates ($\mathcal{T}$-\textsc{Solve.Certificate}). 


\begin{algorithm}[t]
\caption{\textsc{Imhotep-SMT}}
\begin{algorithmic}[1]
\STATE status := \verb+UNSAT+;
\STATE $\phi_B :=  \bigwedge_{i = 1}^p  \left(\lnot b_i  \Rightarrow c_i \right) \bigwedge \left( \sum_{i = 1}^p b_i \le \overline{s} \right)
$;
\WHILE{status == \texttt{UNSAT}}
	\STATE $(b,c) :=$ \textsc{SAT-Solve}$( \phi_B )$;
	\STATE (status, $x$) := $\mathcal{T}$-\textsc{Solve.Check}$(\overline{\supp}(b))$;
	\IF {status == \texttt{UNSAT}}
		\STATE $\phi_{\text{cert}}$ := $\mathcal{T}$-\textsc{Solve.Certificate}$(b, x)$;
		\STATE $\phi_B := \phi_B \land \phi_{\text{cert}}$;
	\ENDIF
\ENDWHILE
\STATE \textbf{return} $\eta = (x,b)$;
\end{algorithmic}
\label{alg:smt}
\end{algorithm}

\subsection{Satisfiability Checking}
It follows from the $2\overline{s}$-sparse flatness property discussed in Section~\ref{sec:problem}, that  for a given assignment of the Boolean variables $b$, with $\vert\supp(b)\vert \le \overline{s}$, 
the remaining $p-\overline{s}$ sensors define a flat output as:
$$ y^{(t)}_{\mathcal{I}}, \quad \ldots, \quad y^{(t-\tau+1)}_{\mathcal{I}}$$
where $\mathcal{I} = \overline{\supp(b)}$.
The next step is to use the flat output in order to calculate the estimate \mbox{$x = \alpha\left(y_\mathcal{I}^{(t)}, \ldots, y_\mathcal{I}^{(t-\tau+1)}\right)$}. Finally, we evaluate if the condition:
\begin{equation}\label{eq:lmq}
\norm{Y_{\overline{\supp(b)}} - {H}_{u,{\overline{\supp(b)}}} (x)}^2 = 0
\end{equation}
is satisfied. This procedure is summarized in Algorithm~\ref{alg:check}.

\begin{algorithm}[ht]
\caption{$\mathcal{T}$\textsc{-Solve.Check}$(\mathcal{I})$}
\begin{algorithmic}[1]
\STATE \textbf{Construct the state estimate:}\\ $\qquad x:= \alpha\left(y_\mathcal{I}^{(t)}, \ldots, y_\mathcal{I}^{(t-\tau+1)}\right)$ 
\IF{ $\norm{Y_{\mathcal{I}} - {H}_{u,{\mathcal{I}}(x)}} == 0$} \label{line:check}
	\STATE status := \texttt{SAT};
\ELSE
	\STATE status := \texttt{UNSAT};
\ENDIF
\STATE \textbf{return} (status, $x$)
\end{algorithmic}
\label{alg:check}
\end{algorithm}

\subsection{Generating Succinct UNSAT Certificates}

Whenever $\mathcal{T}$\textsc{-Solve.Check} provides \texttt{UNSAT},  the na\"{i}ve certificate can always be generated as in~\eqref{eq:trivcert}.
%
%
However, such trivial certificate does not provide much information, since it only excludes the current assignment from the search space, and can lead to exponential execution time, as reflected by the following proposition.

\begin{proposition}
Let the linear dynamical system $\Sigma_a$ defined in~\eqref{eq:system_attack} be $2\overline{s}$-sparse observable. Then, Algorithm~\ref{alg:smt} which uses the trivial UNSAT certificate $\phi_{\text{triv-cert}}$ in~\eqref{eq:trivcert} returns \mbox{$\eta = (x,b)$} such that:
$$ x^* - x \quad \land \quad \supp(b^*) \subseteq \supp(b),$$
where $x^*$ and $b^*$ are the actual system state and attack indicator vector, as defined in Section~\ref{sec:pform}.  
Moreover, the upper bound on the number of iterations of Algorithm~\ref{alg:smt} is $ \sum_{s = 0}^{\overline{s}} \binom{p}{s}.$
\label{prop:trivial-cert}
\end{proposition}

\begin{proof}
Correctness of Algorithm~\ref{alg:smt} follows directly from the $2\overline{s}$-sparse flatness along with Theorem \ref{th:sse}. The worst case bound on the number of iterations would happen when the solver exhaustively explores all possible combinations of attacked sensors with cardinality less than or equal to $\overline{s}$ in order to find the correct assignment. This is equal to $ \sum_{s = 0}^{\overline{s}} \binom{p}{s}$ iterations.
\end{proof}

%


The generated UNSAT certificate heavily affects the overall execution time of Algorithm~\ref{alg:smt}: the  smaller the certificate, the more information is learnt and the faster is the convergence of the SAT solver to the correct assignment. For example,  a certificate with $b_i = 1$ would identify exactly one attacked sensor at each step. 
Therefore, our objective is to design an algorithm that can lead to more \emph{compact certificates} to enhance the execution time of \textsc{Imhotep-SMT}. To do so, we exploit the specific structure of the secure state reconstruction problem and generate customized, yet stronger, UNSAT certificates. 

First, we observe that the measurements of each sensor \mbox{$Y_i = {H}_{u,i} (x)$} define a set $\mathbb{H}_i \subseteq \mathcal{X} $ as:
$$ \mathbb{H}_i  = \{ x \in \mathcal{X} \; \vert  \; Y_i - {H}_{u,i} (x) = 0\}.$$
It is then straightforward to show the following result.
\begin{proposition}
Let the nonlinear dynamical system $\Sigma_a$ defined in~\eqref{eq:system_attack} be $2\overline{s}$-sparse flat.  
Then, for any set of indices $\mathcal{I} \subseteq \{1, \hdots, p\}$,
the following statements are equivalent:
\begin{itemize}
\item $\mathcal{T}$\textsc{-Solve.Check}$(\mathcal{I})$ returns \texttt{UNSAT},
\item $\bigcap_{i \in \mathcal{I}} \mathbb{H}_i = \emptyset$.
\end{itemize}
\label{prop:geom}
\end{proposition}

The existence of a compact Boolean constraint that explains a conflict is then guaranteed by the following Lemma.

\begin{lemma}
Let the nonlinear dynamical system $\Sigma_a$ defined in~\eqref{eq:system_attack} be $2\overline{s}$-sparse flat. 
If $\mathcal{T}$\textsc{-Solve.Check}$(\mathcal{I})$ is \texttt{UNSAT} for a set $\mathcal{I}$, with $\vert \mathcal{I} \vert > p - 2\overline{s}$,
then there exists a subset $\mathcal{I}_{temp} \subset \mathcal{I}$ with  $\vert \mathcal{I}_{temp} \vert \le p - 2\overline{s} + 1$ such that $\mathcal{T}$\textsc{-Solve.Check}$(\mathcal{I}_{temp})$ is also \texttt{UNSAT}.
\label{prop:minSet}
\end{lemma}
\begin{proof}
Consider any set of sensors $\mathcal{I}' \subset \mathcal{I}$ such that \mbox{$\vert \mathcal{I}' \vert = p - 2\overline{s}$} and $\bigcap_{i \in \mathcal{I}'} \mathbb{H}_i$ is not empty. If such set $\mathcal{I}'$ does not exist, then the result follows trivially. If the set $\mathcal{I}' $ exists, then:
\begin{align*}
\bigcap_{i \in \mathcal{I}'} \mathbb{H}_i \ne \emptyset 
&\Rightarrow \norm{Y_{\mathcal{I}'} - {H}_{u,{\mathcal{I}'}} (x')}^2 = 0 \quad \forall x' \in \bigcap_{i \in \mathcal{I}'} \mathbb{H}_i
\end{align*}
However, since the cardinality of $\mathcal{I}'$ is equal to $p - 2\overline{s}$, it follows from the $2\overline{s}$-sparse flatness that any state \mbox{$x' \in \bigcap_{i \in \mathcal{I}'} \mathbb{H}_i$} is distinguishable using $p - 2\overline{s}$ sensors, for which we conclude that the
intersection $\bigcap_{i \in \mathcal{I}'} \mathbb{H}_i$ is a single point, named $x'$.
Now, since  $\mathcal{T}$\textsc{-Solve.Check}$(\mathcal{I})$  is \texttt{UNSAT},  it follows from Proposition~\ref{prop:geom} that:
\begin{align*}
\bigcap_{i \in \mathcal{I}} \mathbb{H}_i = \emptyset &\Rightarrow 
\bigcap_{i \in \mathcal{I}'} \mathbb{H}_i \cap \bigcap_{i \in \mathcal{I}\setminus\mathcal{I}'} \mathbb{H}_i  = \emptyset 
\{x'\} \cap \bigcap_{i \in \mathcal{I}\setminus\mathcal{I}'} \mathbb{H}_i  = \emptyset,
\end{align*}
which in turn implies that there exists at least one sensor \mbox{$i \in \mathcal{I}\setminus\mathcal{I}'$} such that its set $\mathbb{H}_i$ does not include the point $x'$. Now, we define $\mathcal{I}_{temp}$ as $\mathcal{I}_{temp} = \mathcal{I}' \cup i$ and note that $\vert \mathcal{I}_{temp} \vert \le p - 2\overline{s} + 1$, 
which concludes the proof.
\end{proof}

 Based on the intuition in the proof of Lemma \ref{prop:minSet}, our algorithm works as follows. First, we construct the set of indices  $\mathcal{I}'$
 by picking any random set of $p - 2\overline{s}$ sensors. We  then search for one additional sensor $i'$ which can lead to a conflict with the sensors indexed by $\mathcal{I}'$. To do this, we call $\mathcal{T}$\textsc{-Solve.Check} by passing the set $\mathcal{I}_{temp} := \mathcal{I}' \cup i'$ as an argument. If the check returned \texttt{SAT}, then we label these sensors as ``non-conflicting'' and we repeat the same process by replacing the sensor indexed by $i'$
 with another sensor until we reach a conflicting set. It then follows from Lemma \ref{prop:minSet} that this process terminates revealing a collection of $p -2\overline{s} + 1$ conflicting sets. Once this collection is discovered, we stop by generating the following, more compact, certificate:
$$ \phi_{\text{cert}}:= \sum_{i \in \mathcal{I}_{temp}} b_i \ge 1.$$

Although the prescribed process will always terminate regardless of the selection of the initial set $\mathcal{I}'$ or the order followed to select $i'$, the execution time may change.  In Algorithm~\ref{alg:certificate1}, we implement a heuristic for the selection of the initial set $\mathcal{I}'$ and the succeeding indexes, inspired by the strategy we have adopted in the context of linear systems~\cite{Yasser_SMT}.
We are now ready to state the main result of this section.
\begin{theorem}
Let the nonlinear dynamical system $\Sigma_a$ defined in~\eqref{eq:system_attack} be $2\overline{s}$-sparse flat. Then, Algorithm~\ref{alg:smt} using the conflicting UNSAT certificate $\phi_{\text{cert}}$ in Algorithm~\ref{alg:certificate1} returns \mbox{$\eta = (x,b)$} such that:
$$ x^* - x \quad \land \quad \supp(b^*) \subseteq \supp(b)$$
%
Moreover, the upper bound on the number of iterations of Algorithm~\ref{alg:smt} is 
$\binom{p}{p - 2\overline{s}+1}$.
\label{prop:certifcate1}
\end{theorem}

\begin{proof}
Correctness follows from Theorem \ref{th:sse} along with the $2\overline{s}$-flatness condition. The upper bound on the number of iterations of Algorithm \ref{alg:smt} can be derived as follows. First,  Lemma \ref{prop:minSet} ensures that each certificate $\phi_{\text{cert}}$ has at most $p - 2\overline{s} + 1$ variables. Since we know that the algorithm always terminates, the worst case would then happen when the solver exhaustively generates all the conflicting sets of cardinality $p - 2\overline{s} + 1$. This leads to a number of iterations equal to
$\binom{p}{p - 2\overline{s}+1}.$
\end{proof}


\begin{algorithm}[!t]
\caption{$\mathcal{T}\textsc{-Solve.Certificate}({\mathcal{I}}, x)$}
\begin{algorithmic}[1]
\STATE \textbf{Compute normalized residuals} 
\STATE $\quad r := \bigcup_{i \in  \mathcal{I}} \left \{r_i \right \}$, \\$ \qquad r_i := \norm{Y_i -  {H}_{u,i}(x) }^2/\norm{\nabla_x {H}_{u,i}}^2, \; i \in { \mathcal{I}}$;
\STATE \textbf{Sort the residual variables} 
\STATE $\quad r\_sorted := \text{sortAscendingly}(r)$;
\STATE \textbf{Pick the index corresponding to the maximum residual} 
\STATE $\quad \mathcal{I}\_max\_r :=  \text{Index}(r\_sorted_{\{\vert \mathcal{I} \vert, \vert \mathcal{I} \vert - 1,\ldots,  p - 2\overline{s} + 1\}})$;
\STATE $\quad \mathcal{I}\_min\_r :=  \text{Index}(r\_sorted_{\{1, \ldots, p - 2\overline{s} \} })$;
\STATE \textbf{Search linearly for the UNSAT certificate}
\STATE $\quad \text{status = \texttt{SAT}}; \quad \text{counter} = 1;$
\STATE $\quad \mathcal{I}\_temp :=  \mathcal{I}\_min\_r \cup \mathcal{I}\_max\_r_{counter}$;
\WHILE{status == \texttt{SAT}}
	\STATE (status, $x$) $:=  \mathcal{T}\textsc{-Solve.Check}(\mathcal{I}\_temp)$;
	\IF{status == \texttt{UNSAT}}
		\STATE $\phi_{\text{cert}}:= \sum_{i \in \mathcal{I}\_temp} b_i \ge 1$;
	\ELSE
		\STATE counter := counter  + 1;
		\STATE $\mathcal{I}\_{temp} :=  \mathcal{I}\_min\_r \cup \mathcal{I}\_max\_r_{counter}$;
	\ENDIF
\ENDWHILE
\STATE \textbf{return} $\phi_{\text{cert}}$
\end{algorithmic}
\label{alg:certificate1}
\end{algorithm}

\section{Case Study: Securing a Quadrotor Mission}
\label{sec:results}

%

We demonstrate the effectiveness of our detection algorithm by applying it to  a waypoint navigation mission for a quadrotor unmanned aerial vehicle (UAV), in which the UAV needs to cross a workspace from a starting point to a desired goal. The dynamical model of the quadrotor and its controller, based on~\cite{michael2010grasp, lupashin2010simple}, are summarized below.

\subsection{Dynamical Model}
The dynamical model of the quadrotor consists of twelve states which are:
\begin{itemize}
\item $\mathbf{p} = (p_x, p_y, p_z)$ is the quadrotor center of mass in the inertial frame $W$ (world frame).
\item $\mathbf{v} = (v_x, v_y, v_z)$ is the quadrotor linear velocity.
\item $\mathbf{\vartheta} = (\vartheta_\phi, \vartheta_\theta, \vartheta_\psi)$ is the quadrotor orientation (or attitude) expressed in terms of Euler angles: roll, pitch, yaw.
\item $\mathbf{\omega} =(\omega_{\phi}, \omega_{\theta},  \omega_{\psi} )$ is the quadrotor angular velocity.
\end{itemize}


%

The quadrotor is equipped with four motors. The thrust produced by the $i$th motor is denoted by $F_i$ and is directly proportional to the rotor speed. The resulting vertical force is denoted by $u_1$ and is equal to:
$$ u_1 = F_1 + F_2 + F_3 + F_4 - mg$$
where $m$ is the mass and $g$ is the gravitational acceleration. The
motors' thrust also induce three moments, on the quadrotor center of mass, denoted by $u_2, u_3, u_4$. These moments can be computed from the thrusts as follows:
%
%
%
%
\begin{equation*}\label{eq:u2}
\left(%
\begin{array}{c}
  u_2 \\
  u_3\\
  u_4 \\
\end{array}%
\right) =\left(%
\begin{array}{c c c c}
   l_{1x} & -l_{2x} & l_{3x} & -l_{4x}\\
  -l_{1y} & l_{2y} & l_{3y} & -l_{4y}\\
  \mu &\mu & -\mu &-\mu\\
\end{array}%
\right) \left(%
\begin{array}{c}
  F_1 \\
  F_2\\
  F_3 \\
  F_4 \\
\end{array}%
\right) 
\end{equation*}
where 
$l_{ix}$ and $l_{iy}$ are the $x$ and $y$ components of the distance between the motor $i$ and the center of the quadrotor, respectively, and $\mu$ is a constant representing the relationship between lift and drag. 

\subsection{Controller}
The controller is derived by linearizing the equations of motion and the motor models at an operating point that corresponds to the nominal hover state, i.e.~$\mathbf{p} = (p_x, p_y, p_z)$, $\mathbf{\vartheta} = (0, 0,  \psi)$, $\mathbf{v} = (0, 0, 0)$ and $\mathbf{\omega} = (0, 0, 0)$ under the assumptions of small roll and pitch angles, which leads to $\cos(\vartheta_\phi) = \cos(\vartheta_\theta) \approx 1$, $\sin(\vartheta_\phi) \approx \phi$, $\sin(\vartheta_\theta) \approx \vartheta_\theta$. The nominal values for the inputs at hover are $u_1 = mg$, $u_2 = u_3 = u_4 = 0$.

To control the quadrotor to follow a desired trajectory, we use a two-level decoupled control scheme consisting of a low-level attitude control, which usually runs at $1$KHz, and a high-level position control running at a lower rate $50$Hz.
%
%
The position control is used to track a desired trajectory characterized by $\mathbf{p}^{\textrm{des}}$ and $\mathbf{\vartheta}^{\textrm{des}}$. Using a PID feedback controller
 we can then control the position and velocity of the quadrotor to maintain the desired trajectory. 
%
Similarly, we can realize the attitude control using PD controllers.

\subsection{Securing the Quadrotor Trajectory}
The quadrotor is equipped with a GPS measuring the position vector and two inertial measurement units (IMUs), whose outputs are fused to generate an estimate for the body angular and linear velocities. We numerically simulate the model of the quadrotor and the controller. In our scenario, the quadrotor goal is to takeoff vertically and then move along a square trajectory. However, one of the IMU's output, the vertical velocity sensor, is attacked by injecting a sinusoidal signal on top of the actual sensor readings. As shown in Fig.~\ref{fig:exp} (bottom), the attack is injected after the quadrotor has completed the takeoff maneuver, and only along two parallel sides of the whole square trajectory.  

To implement the secure state reconstruction algorithm, \textsc{Imhotep-SMT} uses an approximate discretized model of the plant along with the sensor measurements. To discretize the model we use the same sampling time ($T_s = 20$~ms) of the controller. Moreover, we adopt a first-order forward Euler approximation scheme which preserves the differential flatness of the original system.
To accommodate the model mismatch due to the discrete approximation, as well as round-off errors, we replace the condition in line~\ref{line:check} of Algorithm \ref{alg:check} with $\norm{Y_{\mathcal{I}} - {H}_{u,{\mathcal{I}}}(x)} \le  \epsilon,$ where we set $\epsilon$ to $0.1$ in our experiments. 

Fig.~\ref{fig:nosmt} (top) shows the effect of the attack when the quadrotor operates without secure state reconstruction algorithm. As evident from the two corners of the square trajectory corresponding to coordinates $(2.5,0,1)$ and $(0,2.5,1)$, the injected attack harmfully impairs the stability of the quadrotor, due to incorrect state reconstruction, as shown in Fig.~\ref{fig:nosmt} (middle). 
Fig.~\ref{fig:smt} (top) shows instead the trajectory of the quadrotor when operated using \textsc{Imhotep-SMT} to perform secure state reconstruction. The estimation error on $v_z$ produced by $\textsc{Imhotep-SMT}$ is in the order of $10^{-2}$ m/s, 
and always bounded, where the bound depends on the error due to mismatch between the model used for estimation and the actual quadrotor dynamics (the controller is designed to be robust against bounded perturbations).
The state and the support of the attack are correctly estimated also in the presence of model mismatch: the quadrotor is able to follow the required trajectory and achieve its goal. 
Finally, the average execution time of $16.1$ ms (smaller than the  $20$ms sampling time of the position controller) on an Intel Core i7 3.4-GHz processor with 8~GB of memory, is compatible with several real-time applications. 


\begin{figure*}[!t]
\centering
\begin{tabular}{c|c}
\hspace{-9mm}
\subfigure[Quadrotor under attack
]{\label{fig:nosmt}
\begin{tabular}{c}
{\includegraphics[width=0.47\textwidth]{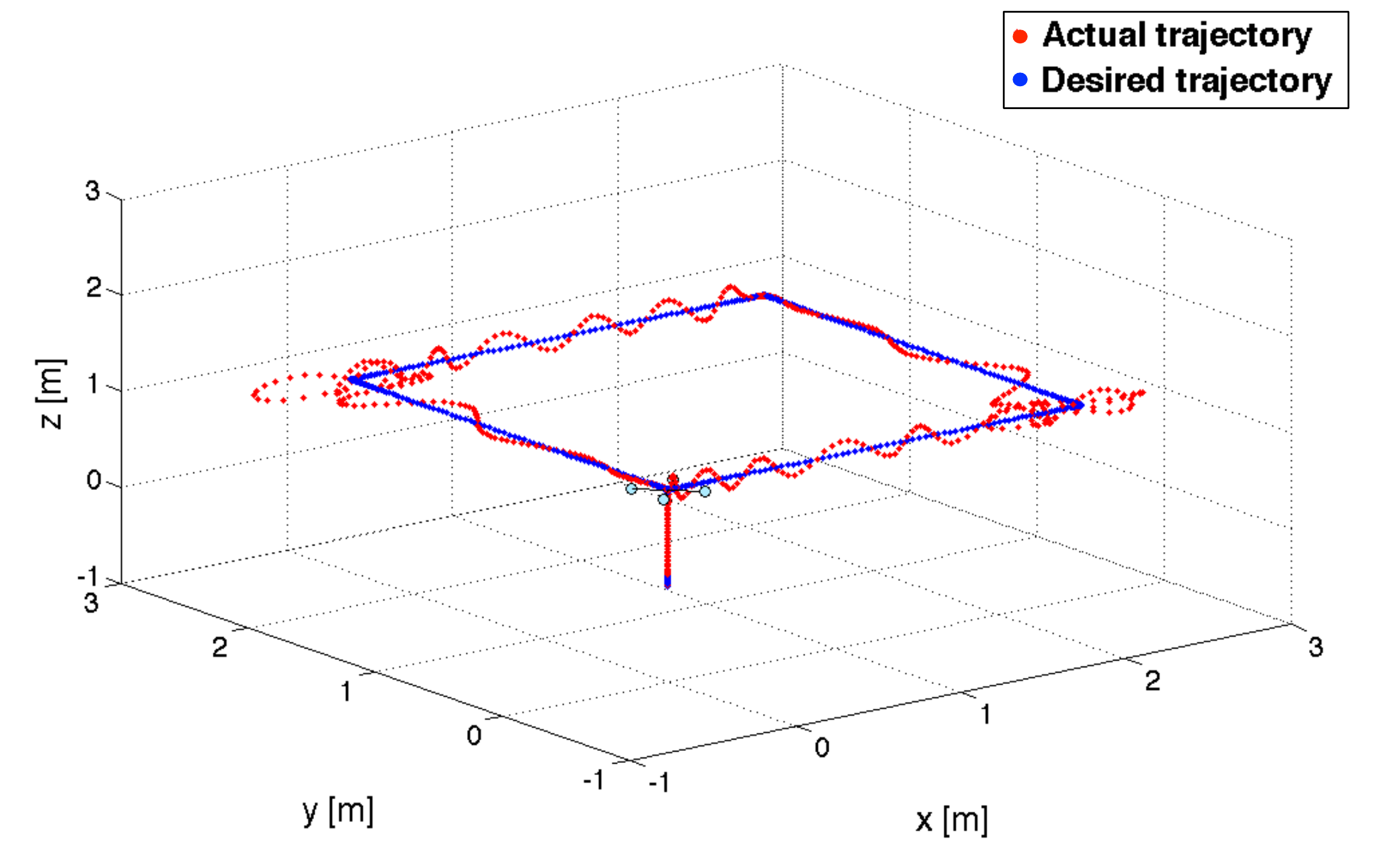} 
}\\
{\includegraphics[width=0.44\textwidth]{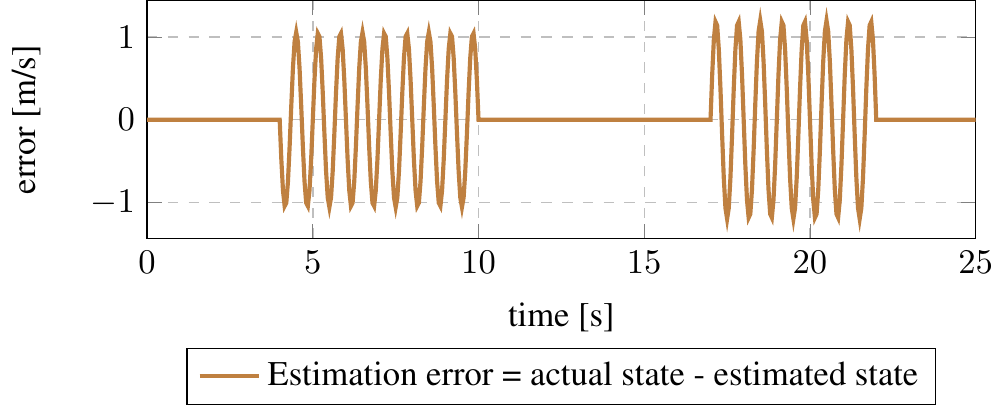} 
}\\
{\includegraphics[width=0.44\textwidth]{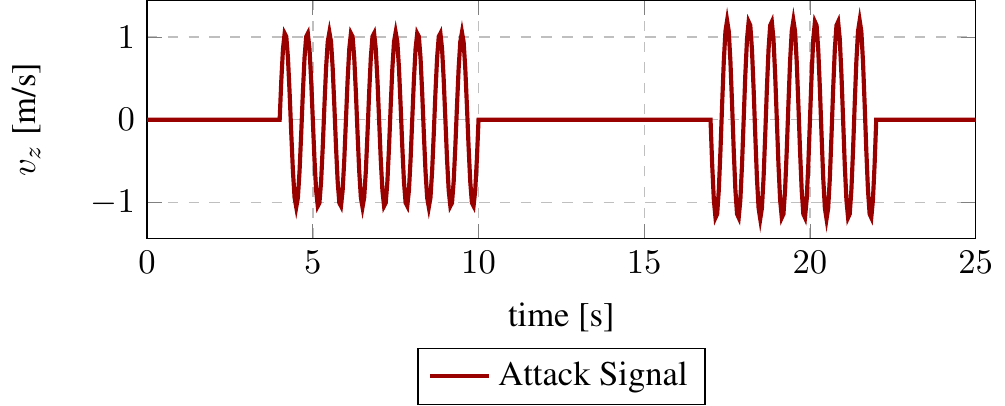} 
}
\end{tabular}
}
&
\subfigure[Quadrotor under attack - \textsc{Imhotep-SMT}
]{\label{fig:smt}
\begin{tabular}{c}
{\includegraphics[width=0.48\textwidth]{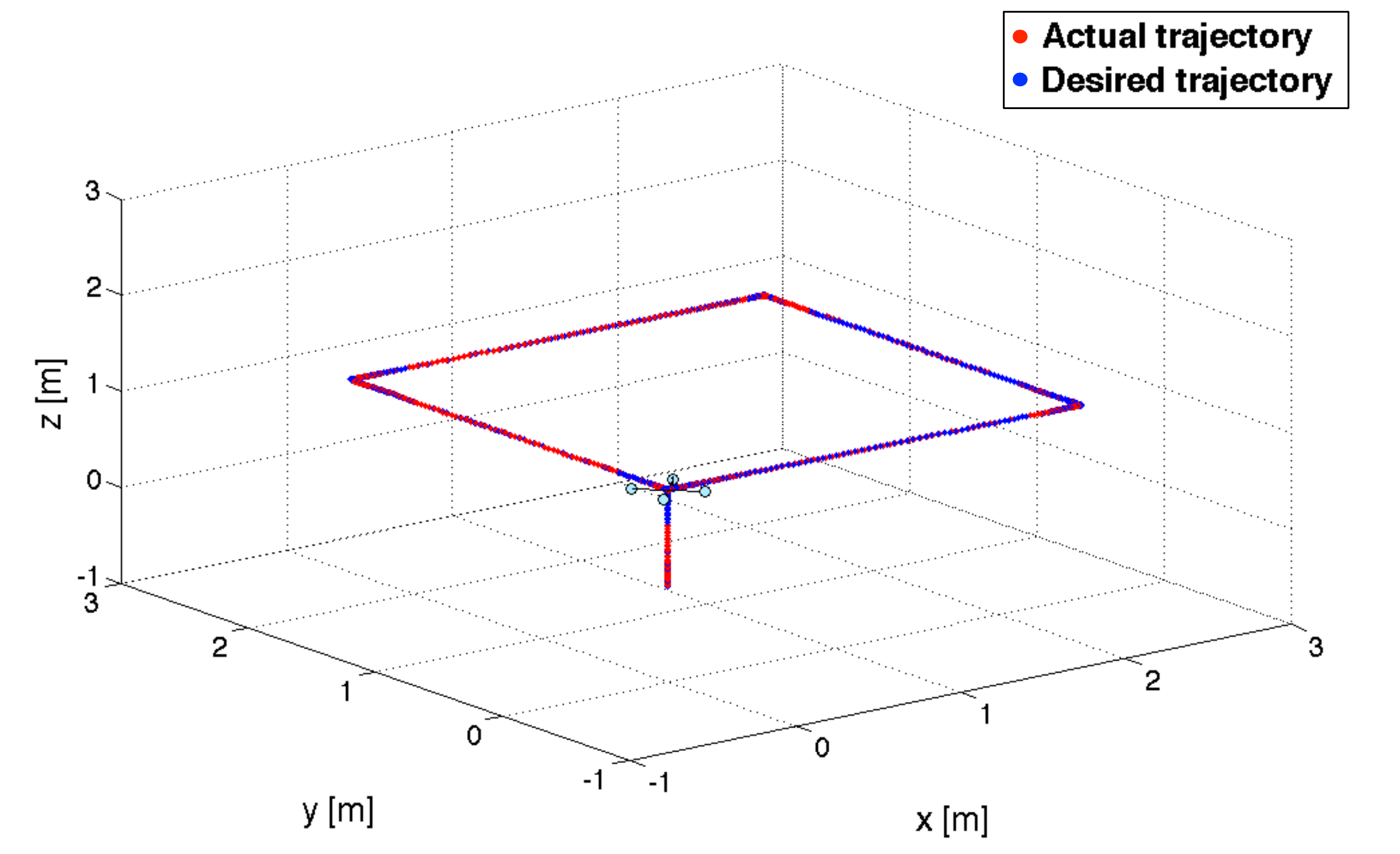} 
}\\
{\includegraphics[width=0.44\textwidth]{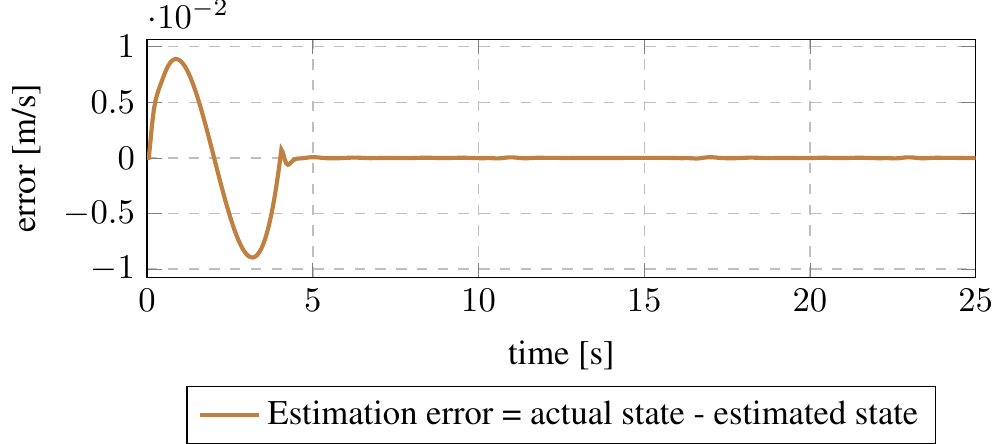} 
}\\
{\includegraphics[width=0.44\textwidth]{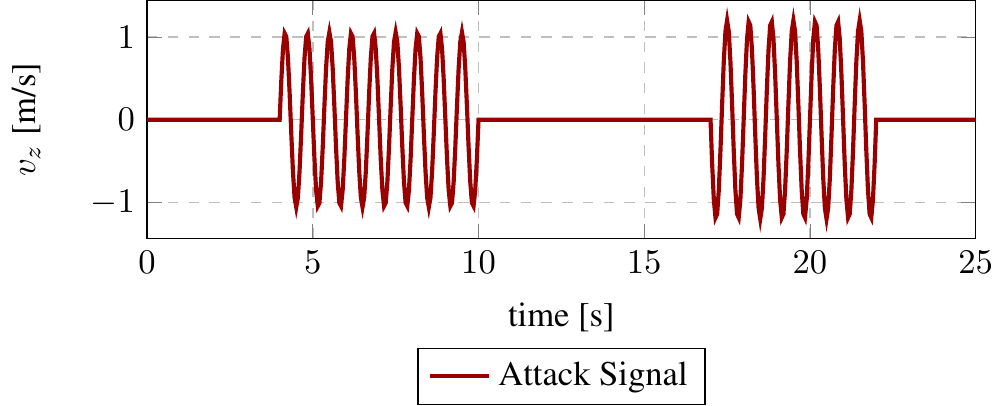} 
}
\end{tabular}
}
\end{tabular}
\caption{Performance of the quadrotor under sensor attack: (a)  without secure state reconstruction algorithm, and (b) with \textsc{Imhotep-SMT}. For each scenario, we show the quadrotor  trajectory (top), the estimation error on the quadrotor vertical velocity (middle), and the attack signal (bottom).}
\label{fig:exp}
\vspace{-4mm}
\end{figure*}
\section{Conclusions}
\label{sec:conclusion}

We have investigated, for the first time, the state reconstruction problem from a set of adversarially
attacked sensors for a class of nonlinear systems, namely differentially flat systems. Given an upper bound $\overline{s}$ on the number of  attacked sensors, we showed that $2\overline{s}$-observability is a necessary and sufficient condition for reconstructing the state in spite of the attack. We have then proposed a Satisfiability Modulo Theory based detection algorithm for differentially flat systems, by extending our previous results, reported in \cite{Yasser_SMT_ACC, Yasser_SMT}, to differentially flat systems. Numerical results show that secure state estimation in complex nonlinear systems, such as in waypoint navigation of a quadrotor under sensor attacks, can indeed be performed with our algorithm in an accurate and efficient way.

\bibliographystyle{IEEEtran}
\bibliography{bibliography2}

\end{document}